\documentclass[a4paper,12pt,final]{amsart}
\usepackage{times,a4wide,mathrsfs,amssymb,amsmath,amsthm}

\newcommand{\C}{\mathbb{C}}

\newcommand{\QQ}{\mathbb{Q}}
\newcommand{\NN}{\mathbb{N}}
\newcommand{\PP}{\mathbb{P}}

\newcommand{\OO}{\mathcal O}

\newcommand{\CC}{\mathcal C}

\newcommand{\FF}{\mathcal F}

\newcommand{\wt}{\widetilde}
\newcommand{\rom}{\romannumeral}

\newtheorem{theorem}{Theorem}[section]
\newtheorem{claim}[theorem]{Claim}
\newtheorem{lemma}[theorem]{Lemma}

\newtheorem{proposition}[theorem]{Proposition}
\newtheorem{conjecture}[theorem]{Conjecture}

\newtheorem{definition}[theorem]{Definition}
\newtheorem{convention}{Conventions}

\newtheorem{problem}[theorem]{Problem}

\newtheorem{nonumbering}{Theorem}

\newtheorem{nonumberingt}{Acknowledgements}

\begin{document}
\author[Robert Laterveer]
{Robert Laterveer}

\address{Institut de Recherche Math\'ematique Avanc\'ee,
CNRS -- Universit\'e 
de Strasbourg,\
7 Rue Ren\'e Des\-car\-tes, 67084 Strasbourg CEDEX,
FRANCE.}
\email{robert.laterveer@math.unistra.fr}

\title{A remark on Beauville's splitting property}

\begin{abstract} Let $X$ be a hyperk\"ahler variety. Beauville has conjectured that a certain subring of the Chow ring of $X$ should inject into cohomology. This note proposes a similar conjecture for the ring of algebraic cycles on $X$ modulo algebraic equivalence: a certain subring (containing divisors and codimension $2$ cycles) should inject into cohomology.
We present some evidence for this conjecture.
\end{abstract}


\maketitle

\section{Introduction}

   For a smooth projective variety $X$ over $\C$, let $A^i(X)=CH^i(X)_{\QQ}$ denote the Chow group of codimension $i$ algebraic cycles modulo rational equivalence with $\QQ$--coefficients. Intersection product defines a ring structure on $A^\ast(X)=\oplus_i A^i(X)$. In the case of $K3$ surfaces, this ring structure has a remarkable property:

\begin{theorem}[Beauville--Voisin \cite{BV}]\label{K3} Let $S$ be a $K3$ surface. Let $D_i, D_i^\prime\in A^1(S)$ be a finite number of divisors. Then
  \[ \sum_i D_i\cdot D_i^\prime=0\ \ \ \hbox{in}\ A^2(S)_{}\ \Leftrightarrow\ \sum_i [D_i]\cup [D_i^\prime]=0\ \ \ \hbox{in}\ H^4(S,\QQ)\ .\]
  \end{theorem}

In the wake of this result (combined with results concerning the Chow ring of abelian varieties \cite{Beau}), Beauville has asked which varieties have behaviour similar to theorem \ref{K3}. This is the problem of determining which varieties verify the ``splitting property'' of \cite{Beau3}. We briefly state this problem here as follows:

\begin{problem}[Beauville \cite{Beau3}]\label{prob} Find a class $\CC$ of varieties (containing $K3$ surfaces, abelian varieties and hyperk\"ahler varieties), such that for any $X\in\CC$, the Chow ring of $X$ admits a multiplicative bigrading $A^\ast_{(\ast)}(X)$, with
  \[ A^i(X)=\bigoplus_{j= 0}^i A^i_{(j)}(X)\ \ \ \hbox{for\ all\ }i\ .\] 
This bigrading should split the conjectural Bloch--Beilinson filtration, in particular 
  \begin{equation}\label{ideal} A^i_{hom}(X)= \bigoplus_{j\ge 1} A^i_{(j)}(X)\ .\end{equation}
  \end{problem}
  
 This question is hard to answer in practice, since we do not have the Bloch--Beilinson filtration at our disposal. However, as noted by Beauville, the class $\CC$ has some nice properties that can be tested in practice. In particular, the conjecture that hyperk\"ahler varieties are in $\CC$ leads to the so--called {\em weak splitting property conjecture\/}, which is the following falsifiable statement: 
 
 \begin{conjecture}[Beauville \cite{Beau3}, Voisin \cite{V12}]\label{split} Let $X$ be a hyperk\"ahler variety, and let $D^\ast(X)\subset A^\ast(X)$ be the $\QQ$--subalgebra generated by divisors and Chern classes. 
 The cycle class map induces an injection
   \[ D^i(X)\ \hookrightarrow\ H^{2i}(X,\QQ)\ \]
   for all $i$. 
  \end{conjecture} 
  
  (cf. \cite{V12}, \cite{V13}, \cite{Vo}, \cite{V14}, \cite{LFu2}, \cite{Rie2}, \cite{Y}, \cite{CP}  for extensions and partial results concerning conjecture \ref{split}.)

An interesting novel approach to problem \ref{prob} (as well as a reinterpretation of theorem \ref{K3}) is provided by the concept of {\em multiplicative Chow--K\"unneth decomposition\/}, giving rise to unconditional constructions of a bigraded ring structure on the Chow ring of certain varieties \cite{SV}, \cite{V6}, \cite{SV2}, \cite{FTV}. 
(The bigrading constructed in these works should be seen as a {\em candidate\/} for the (only ideally existing) bigrading evoked in problem \ref{prob}; in particular, it is not known whether property (\ref{ideal}) holds for these candidates.)

This note does not directly address problem \ref{prob} or conjecture \ref{split}. Instead, our aim is to propose a modified version of conjecture \ref{split}. The modification consists in considering the groups $B^\ast(X)$ of cycles with $\QQ$--coefficients modulo {\em algebraic equivalence\/}. For any $X\in\CC$ (in particular, for a hyperk\"ahler variety), the conjectural bigrading $A^\ast_{(\ast)}(X)$ is expected to be of motivic origin (i.e., induced by a Chow--K\"unneth decomposition). As such, one expects the bigrading to pass to algebraic equivalence and
induce a bigrading $B^\ast_{(\ast)}(X)$. Now, it has been conjectured that (for any smooth projective variety) the deepest level $F^i A^i(X)$ of the conjectural Bloch--Beilinson filtration should be algebraically trivial \cite{J3}, and so $B^2_{(2)}(X)=0$. For a hyperk\"ahler variety, one expects that also $B^2_{(1)}(X)=0$ (this is clear when $X$ is of $K3^{[n]}$ type; for general hyperk\"ahler varieties, one can reason as in the proof of proposition \ref{lessthan0} below), and so conjecturally  
  \[ B^2(X)= B^2_{(0)}(X)\ .\]
  This leads to the following variant of conjecture \ref{split}:
                
  \begin{conjecture}\label{conjB} Let $X$ be a hyperk\"ahler variety. Let $E^\ast(X)\subset B^\ast(X)$ be the $\QQ$--subalgebra generated by $B^1(X)$, $B^2(X)$ and the Chern classes. The cycle class map induces injections
    \[ E^i(X)\ \hookrightarrow\ H^{2i}(X,\QQ)\ \ \ \forall i\ .\]
   \end{conjecture}    
     
    Here is some evidence we have found for conjecture \ref{conjB}: 
    
   \begin{nonumbering}[=theorem \ref{HK4}] Let $X$ be either 
  
  \noindent
  (\rom1) a Hilbert scheme $X=S^{[2]}$, where $S$ is a projective $K3$ surface, or
  
  \noindent
  (\rom2) a Fano variety of lines $X=F(Y)$, where $Y\subset\PP^5(\C)$ is a very general cubic fourfold.     
  
  The cycle class map induces an injection
  \[ E^3(X)\ \hookrightarrow\ H^6(X,\QQ)\ .\] 
      \end{nonumbering}
  
  \begin{nonumbering}[=theorem \ref{kummer}] Let $X=K_m(A)$ be a generalized Kummer variety of dimension $2m$. The cycle class map induces an injection
    \[ E^{2m-1}(X)\ \hookrightarrow\ H^{4m-2}(X,\QQ)\ .\]
  \end{nonumbering}
  
 Our evidence is, alas, restricted to $1$--cycles. The reason for this restriction is that in proving theorems \ref{HK4} and \ref{kummer}, we rely on the bigrading of the Chow ring of $X$ constructed unconditionally in \cite{SV} resp. \cite{FTV}. In both cases, it is not known whether the bigrading satisfies property (\ref{ideal}) for all $i$ (this is only known for $i\ge\dim X-1$).

 \begin{convention} In this article, the word {\sl variety\/} will refer to a reduced irreducible scheme of finite type over $\C$. A {\sl subvariety\/} is a (possibly reducible) reduced subscheme which is equidimensional. 

{\bf All groups of cycles will be with rational coefficients}: we will denote by $A_j(X)$ the Chow group of $j$--dimensional algebraic cycles on $X$ with $\QQ$--coefficients; for $X$ smooth of dimension $n$ we will write $A^i(X):=A_{n-i}(X)$. Likewise, we will write $B_j(X)$ for the group of
$j$--dimensional algebraic cycles on $X$ with $\QQ$--coefficients, modulo algebraic equivalence, and $B^i(X):=B_{n-i}(X)$ for $X$ smooth.

The notations $A^i_{hom}(X)$, $A^i_{alg}(X)$ will be used to indicate the subgroups of homologically trivial, resp. algebraically trivial cycles. Likewise, we write $B^i_{hom}(X)$ for what is commonly  known as the Griffiths group of $X$.

We will write $H^j(X)$ 
to indicate singular cohomology $H^j(X,\QQ)$.
\end{convention}

  \section{Some hyperk\"ahler fourfolds}
  
  \begin{theorem}\label{HK4} Let $X$ be either 
  
  \noindent
  (\rom1) a Hilbert scheme $X=S^{[2]}$, where $S$ is a projective $K3$ surface, or
  
  \noindent
  (\rom2) a Fano variety of lines $X=F(Y)$, where $Y\subset\PP^5(\C)$ is a very general cubic fourfold.
  
  Let $E^\ast(X)\subset B^\ast(X)$ be the $\QQ$--subalgebra generated by $B^1(X)$, $B^2(X)$ and the Chern classes. Then the cycle class map induces an injection
    \[ E^i(X)\ \hookrightarrow\ H^{2i}(X)\ \ \ \hbox{for}\ i\ge 3\ .\]
   \end{theorem}
   
  \begin{proof} Since algebraic and homological equivalence coincide for $0$--cycles, the $i=4$ case is trivially true. The interesting part of the statement is thus only the injectivity of $E^3(X)\to H^6(X)$.
  
  In both cases (\rom1) and (\rom2), there exists a bigraded ring structure $A^\ast_{(\ast)}(X)$ induced by the Fourier transform constructed in \cite{SV}. In both cases,
  the bigrading is also described by the action of a Chow--K\"unneth decomposition, and therefore
   the ring $B^\ast(X)$ inherits a bigrading $B^\ast_{(\ast)}(X)$. The Chern classes of $X$ are in $A^\ast_{(0)}(X)$ (in case (\rom1), this is \cite[Theorem 2]{V6}; in case (\rom2), this follows from the fact that the Chern classes are polynomials in the classes labelled $l\in A^1(X),c\in A^2(X)$ in \cite{SV} (coming from the tautological bundle on the Grassmannian), and it is known that $c\in A^2_{(0)}(X)$ \cite[Theorem 21.9(\rom3)]{SV}).
  
  The theorem now follows from the following claim:
  
  \begin{claim}\label{only0} One has $B^2(X)=B^2_{(0)}(X)$.
  \end{claim}
  
  Indeed: the claim, combined with the above remarks, implies that $E^\ast(X)\subset B^\ast_{(0)}(X)$. But we know (lemma \ref{inj} below) that $B^3_{(0)}(X)\to H^6(X)$ is injective, and so theorem \ref{HK4} is proven.
  
  The claim follows from the fact, proven by Shen--Vial \cite[Theorems 2.2 and 2.4]{SV}, that there exists a correspondence $L\in A^2(X\times X)$ with the property that
    \[  A^2_{(2)}(X)=  L_\ast A^4_{(2)}(X)\ .\]
    Indeed, any $0$--cycle $a\in A^4_{(2)}(X)$ is (homologically trivial hence) algebraically trivial. As algebraic equivalence is an adequate equivalence relation, it follows that $L_\ast(a)$ is algebraically trivial and so
    \[ A^2_{(2)}(X)\ \subset\ A^2_{alg}(X)\ .\]
    This proves the claim:
    \[ B^2(X)= A^2(X)/ A^2_{alg}(X) = \bigl( A^2_{(0)}(X)\oplus A^2_{(2)}(X)\bigr)/   A^2_{alg}(X) =     A^2_{(0)}(X)/ A^2_{alg}(X) =B^2_{(0)}(X)\ .\]

It only remains to prove the following lemma:

 \begin{lemma}[Shen--Vial \cite{SV}]\label{inj} Let $X$ be either 
 
  \noindent
  (\rom1) a Hilbert scheme $X=S^{[2]}$, where $S$ is a projective $K3$ surface, or
  
  \noindent
  (\rom2) a Fano variety of lines $X=F(Y)$, where $Y\subset\PP^5(\C)$ is any smooth cubic fourfold. 
  
  Then $A^3_{(0)}(X)\cap A^3_{hom}(X)=0$.
  \end{lemma}
  
  \begin{proof} This is contained in \cite{SV}.  A quick way of proving the lemma is as follows: let $\FF$ be the Fourier transform of \cite{SV}. We have that $a\in A^3(X)$ is in $A^3_{(0)}(X)$ if and only if $\FF(a)\in A^1_{(0)}(X)=A^1(X)$ \cite[Theorem 2]{SV}. Suppose $a\in A^3_{(0)}(X)$ is homologically trivial. Then also $\FF(a)\in A^1(X)$ is homologically trivial, hence $\FF(a)=0$ in $A^1(X)$. But then, using \cite[Theorem 2.4]{SV}, we find that
      \[ {25\over 2} a = \FF\circ \FF(a)=0\ \ \ \hbox{in}\ A^3(X)\ .\]
  \end{proof}
    \end{proof} 
    
   In theorem \ref{HK4}(\rom2), we restricted to {\em very general\/} cubic fourfolds. The reason is that for the Fano variety $X$ of lines on any given smooth cubic fourfold, it is not yet known that the Fourier decomposition $A^\ast_{(\ast)}(X)$ is a bigraded ring structure (cf. \cite[Remark 22.9]{SV}). If we abandon the hypothesis ``very general'', we can obtain a weaker statement:

\begin{definition} Let $Y\subset\PP^5(\C)$ be a smooth cubic fourfold, and let $X$ be the Fano variety of lines on $Y$. One defines
  \[ A^1(X)_{prim}:= P_\ast (A^2(Y)_{prim})\ \ \ \subset\ A^1(X)\ ,\]
  where $P\in A^3(Y\times X)$ is the universal family of lines, and $A^2(Y)_{prim}:=\{c\in A^2(Y)\ \vert\ [c]\in H^4(Y)_{prim}\}$.
  We set $B^1(X)_{prim}=A^1(X)_{prim}$.
\end{definition}

  \begin{proposition}\label{weak} Let $Y\subset\PP^5(\C)$ be any smooth cubic fourfold, and let $X=F(Y)$ be the Fano variety of lines in $Y$. Let $b\in B^3(X)$ be a
  cycle of the form
    \[  b=\sum_{k=1}^r a_k\cdot d_k\ \ \ \in B^3(X)\ ,\]
    where $a_k\in B^2(X)$ and $d_k\in B^1(X)_{prim}$. Then $b$ is algebraically trivial if and only if $b$ is homologically trivial.
   \end{proposition}
   
   \begin{proof} Claim \ref{only0} still applies to $X$, and so the $a_k$ are in $B^2_{(0)}(X)$. As such, they can be lifted to $\bar{a_k}\in A^2_{(0)}(X)$. One knows
   that
     \[ A^2_{(0)}(X)\cdot A^1(X)_{prim}\ \subset\ A^3_{(0)}(X) \]
     \cite[Proposition 22.7]{SV}, and thus $\bar{a_k}\cdot d_k\in A^3_{(0)}(X)$. It follows that $b\in B^3_{(0)}(X)$, and one concludes using lemma \ref{inj}.
    \end{proof}

  \section{Generalized Kummer varieties}        
    
  \begin{theorem}\label{kummer} Let $A$ be an abelian surface, and let $X=K_m(A)$ be a generalized Kummer variety of dimension $2m$.  Let $E^\ast(X)\subset B^\ast(X)$ be the $\QQ$--subalgebra generated by $B^1(X)$, $B^2(X)$ and the Chern classes. The cycle class map induces injections
    \[ E^i(X)\ \hookrightarrow\ H^{2i}(X)\ \]
   for $i\ge 2m-1$. 
    \end{theorem}

 \begin{proof} Thanks to \cite[Theorem 7.9]{FTV}, there exists a multiplicative Chow--K\"unneth decomposition for $X$ and so the Chow ring has a bigrading $A^\ast_{(\ast)}(X)$. 
 Moreover, the Chern classes of $X$ are in $A^\ast_{(0)}(X)$ \cite[Proposition 7.13]{FTV}. Let $B^\ast_{(\ast)}(X)$ denote the induced bigrading modulo algebraic equivalence.
 
 \begin{proposition}\label{lessthan0} We have
   \[ B^2(X)=\bigoplus_{j\le 0} B^2_{(j)}(X)\ .\]
  \end{proposition} 
  
  \begin{proof} One knows that $B^i_{(i)}(X)=0$ for all $i>0$ \cite[Corollary 18]{hard}. It remains to check that $B^2_{(1)}(X)=0$. 
  By definition, we have
    \[ B^2_{(1)}(X) = (\Pi^X_{3})_\ast B^2(X)\ ,\]
    where $\{\Pi^X_j\}$ is the multiplicative Chow--K\"unneth decomposition furnished by \cite{FTV}. Since $B^2_{(1)}(X)\subset B^2_{hom}(X)$, and $\Pi^X_3$ is idempotent, we also have
    \begin{equation}\label{also}      B^2_{(1)}(X) = (\Pi^X_{3})_\ast B^2_{hom}(X)   \ .\end{equation} 
     Next, we observe that (as $X$ is hyperk\"ahler)
  $H^3(X,\OO_X)=0$. Since the generalized Hodge conjecture is known to hold for self--products of abelian surfaces \cite[7.2.2]{Ab}, \cite[8.1(2)]{Ab2}, and generalized Kummer varieties are {\em motivated\/} by abelian surfaces in the sense of \cite{Ara}, the generalized Hodge conjecture is true for generalized Kummer varieties (for the usual Hodge conjecture, this was noted in \cite[Theorem 3.3]{Xu}). In particular, $H^3(X)$ is supported on a divisor $D\subset X$, and $H^{2n-3}(X)$ is supported on a $2$--dimensional subvariety $S\subset X$. Using the Lefschetz $(1,1)$ theorem, one can find a cycle $\gamma\in A^{2m}(X\times X)$ representing the K\"unneth component $\pi^X_3$ and supported on $S\times D$. For dimension reasons, we have
    \[  \gamma_\ast B^2_{hom}(X)=0 \ .\]
    (Indeed, the action of $\gamma$ on $B^2_{hom}(X)$ factors over $B^2_{hom}(\wt{S})=0$, where $\wt{S}\to S$ denotes a desingularization.)
  Applying lemma \ref{equiv} below, this implies that also
   \[    (\Pi^X_{3})_\ast B^2_{hom}(X)=0\ ,\]
   and we are done in view of (\ref{also}).
   
   Here, we have used the following lemma. (The lemma applies to our set--up, because generalized Kummer varieties have finite--dimensional motive \cite{Xu}, \cite{FTV}.)
   
   \begin{lemma}\label{equiv} Let $X$ be a smooth projective variety of dimension $n$, and assume $X$ has finite--dimensional motive. Let $\Pi$ and $\pi\in A^n(X\times X)$ be such that $\Pi$ is idempotent and $\Pi=\pi$ in $H^{2n}(X\times X)$. Then 
   \[  \pi_\ast B^i_{hom}(X)=0\  \Rightarrow  \  \Pi_\ast B^i_{hom}(X)=0\ .\]
    \end{lemma}
    
  \begin{proof} We have
      \[ \Pi-\pi\ \ \in\ A^{n}_{hom}(X\times X)\ .\]
    From Kimura's nilpotence theorem \cite{Kim}, it follows that there exists $N\in \NN$ such that
      \[ (\Pi-\pi)^{\circ N}=0\ \ \ \hbox{in}\ A^{n}(X\times X)\ .\]
      Developing this expression, we obtain
      \[  \Pi=\Pi^{\circ N}= P_1+ P_2+\cdots +P_m\ \ \ \hbox{in}\ A^{n}(X\times X)_{\QQ}\ ,\]  
      where each $P_j$ is a composition of correspondences containing at least one copy of $\pi$. But then (by hypothesis) the right--hand side acts as zero on $B^i_{hom}(X)_{}$, and hence so does the left--hand side.      
  \end{proof}  
  
  This ends the proof of proposition \ref{lessthan0}.     
   \end{proof}

 Proposition \ref{lessthan0}, combined with the fact that the Chern classes are in $B^\ast_{(0)}(X)$, implies that there is an inclusion
   \[ E^\ast(X)\ \subset\ \bigoplus_{j\le 0} B^\ast_{(j)}(X)\ .\] 
  
  Theorem \ref{kummer} follows from this inclusion, combined with the following lemma:
  
  \begin{lemma} Let $X$ be a generalized Kummer variety of dimension $2m$. Then
    \[ B^{2m-1}_{(j)}(X) =0\ \ \ \hbox{for}\ j<0\ ,
   \ \ \ \hbox{and}\ \ B^{2m-1}_{(0)}(X)\cap B^{2m-1}_{hom}(X)=0\ .\]
   \end{lemma}
   
   To prove the lemma, we note that by definition,
     \[  B^{2m-1}_{(j)}(X) = (\Pi^X_{4m-2-j})_\ast B^{2m-1}(X)\ ,\]
     where $\{ \Pi^X_k\}$ is a multiplicative Chow--K\"unneth decomposition \cite{FTV}. Inspecting the construction in \cite{FTV}, one finds that $\Pi^X_{4m-1}=0$ and
     $\Pi^X_{4m}$ is of the form $X\times x$, where $x\in X$. This proves the first statement.
     
     As for the second statement of the lemma, we observe that there exists a cycle $\gamma\in A^{2m}(X\times X)$ representing the K\"unneth component $\pi^X_{4m-2}$ and supported on $X\times S$, where $S\subset X$ is a smooth surface (this is a general fact, for any variety $X$ verifying the Lefschetz standard conjecture $B(X)$, cf. \cite[Theorem 7.7.4]{KMP}). For dimension reasons, we have
     \[  \gamma_\ast B^{2m-1}_{hom}(X)=0\ . \]
     (Indeed, the action of $\gamma$ on $B^{2m-1}_{hom}(X)$ factors over $B^1_{hom}({S})=0$). By lemma \ref{equiv}, this implies that
     \[ (\Pi^X_{4m-2})_\ast  B^{2m-1}_{hom}(X)=0\ . \]   
     On the other hand, $\Pi^X_{4m-2}$ is a projector on $B^{2m-1}_{(0)}(X)$, and so
     \[ B^{2m-1}_{(0)}(X) \cap B^{2m-1}_{hom}(X)= (\Pi^X_{4m-2})_\ast  B^{2m-1}_{hom}(X)\ .\]          
   \end{proof}

\vskip1cm
\begin{nonumberingt} Thanks to all participants of the Strasbourg 2014/2015 ``groupe de travail'' based on the monograph \cite{Vo} for a stimulating atmosphere.
Thanks to Yoyo, Kai and Len for wonderful Christmas holidays. Thanks to the referee for highly pertinent remarks.
\end{nonumberingt}

\end{document}